\documentclass[11pt, reqno]{amsart}

\usepackage{amsmath,amssymb,amsthm, epsfig}
\usepackage{hyperref}
\usepackage{amsmath}
\usepackage{amssymb}
\usepackage{color}
\usepackage{ulem}
\usepackage{epsfig}
\usepackage[mathscr]{eucal}
\usepackage[latin1]{inputenc}

\newtheorem{theorem}{Theorem}
\newtheorem{definition}{Definition}

\newtheorem{lemma}{Lemma}
\newtheorem{proposition}{Proposition}
\newtheorem{corollary}{Corollary}

\date{}
\numberwithin{equation}{section}
\numberwithin{theorem}{section}
\numberwithin{lemma}{section}
\numberwithin{corollary}{section}
\numberwithin{remark}{section} \numberwithin{proposition}{section}
\numberwithin{definition}{section}

\def \R {\mathbb{R}}

\def \dist {\mathrm{dist}}

\def \Lip{\mathrm{Lip}}

\begin{document}

\title[A free boundary optimization problem for the $\infty$-Laplacian]{A free boundary optimization problem\\ for the $\infty$-Laplacian}

\author[R. Teymurazyan]{Rafayel Teymurazyan}
\address{CMUC, Department of Mathematics, University of Coimbra, 3001-501 Coimbra, Portugal.}
\email{rafayel@mat.uc.pt}

\author[J.M. Urbano]{Jos\'e Miguel Urbano}
\address{CMUC, Department of Mathematics, University of Coimbra, 3001-501 Coimbra, Portugal.}
\email{jmurb@mat.uc.pt}

\begin{abstract}
We study a free boundary optimization problem in heat conduction, ruled by the infinity--Laplace operator, with lower temperature bound and a volume constraint. We obtain existence and regularity results and derive geometric properties for the solution and the free boundaries.

\bigskip


\bigskip

\noindent \textbf{AMS Subject Classifications MSC 2010:} 35R35, 35J20, 35J92, 35B65.

\end{abstract}

\maketitle

\section{Introduction}

The goal of this paper is to establish existence, and derive geometric properties, for an optimization problem in heat conduction. The problem may be described in the following way: given a non-negative function $\varphi$ (the temperature profile), our goal is to keep the temperature in a room above $\varphi$, by using heating sources inside the room and insulation material of a certain volume outside the room, in a way that minimizes the energy. We consider the energy driven by the infinity-Laplace operator
$$\Delta_\infty u:=\sum_{i,j=1}^{n}u_{x_i}u_{x_j}u_{x_ix_j},$$
which, despite being too degenerate to realistically represent a physical diffusion process, has been previously used in the context of free boundary problems, for example in \cite{ALT}, where a dead-core problem is considered. We stress that one of the major difficulties when dealing with $\Delta_\infty$ relies precisely on the fact that it diffuses only in the direction of the gradient, which changes point-by-point and depends on the solution itself. It is striking that genuine physical insights can be used to investigate an apparently nonphysical problem with significant mathematical interest.

In mathematical terms, given a smooth bounded domain $\Omega\subset\mathbb{R}^n$, a smooth non-negative function $\varphi:\R^n\rightarrow\R$, compactly supported in $\Omega$, and a positive number $\gamma>0$, we look for a function $u:\R^n\rightarrow\R$ that minimizes
\begin{equation}\label{L}\tag{$P_\infty$}
\Lip(u)\textrm{ in }\mathbb{K}_\infty,
\end{equation}
where 
$$\mathbb{K}_\infty = \left\{ u\in W^{1,\infty}(\R^n)\, | \, u\geq\varphi, \ |\{u>0\}\setminus\Omega| \leq \gamma  \right\}, $$
and such that
$$
\left\{
\begin{array}{rcll}
\Delta_\infty u=0 & \mathrm{in} & \{u>0\}\setminus\Omega & \mathrm{(insulation)},\\ 
\\
\Delta_\infty u\leq0 & \mathrm{in} & \Omega & \mbox{(interior heating)}.
\end{array}
\right.
$$

\noindent Here, $\Lip(u)$ is the Lipschitz constant of $u$
$$
\Lip(u):=\sup_{x,y}\frac{|u(x)-u(y)|}{|x-y|},
$$
$|E|$ is the $n$-dimensional Lebesgue measure of the set $E$, and the relations on $\Delta_\infty u$ are understood in the viscosity sense, according to the next definition.
\begin{definition}\label{d1.1}
A continuous function $u$ is called a viscosity super-solution (resp. sub-solution) of $\Delta_\infty u=0$ if for every $C^2$ function $\phi$ such that $u-\phi$ has a local minimum at the point $x_0$, with $\phi(x_0)=u(x_0)$, we have
$$\Delta_\infty\phi(x_0)\leq0.\quad \textrm{(resp. $\geq$)}$$
A function $u$ is called a viscosity solution if it is both a viscosity super-solution and a viscosity sub-solution.
\end{definition}

The problem arises in the study of best insulation devices but motivations also come from plasma physics or flame propagation, for example. The study of optimal configuration problems started decades ago (see \cite{AAC86, AC81,  ACS87}) and has been developed in recent years to treat optimal design problems ruled by a large class of divergence type operators (see \cite{BMW06, OT06, T05, T10}). The case where, instead of the infinity--Laplacian, one has the standard Laplace operator was studied in \cite{Y16}. Since the fractional Laplacian can be represented as a "Dirichlet to Neumann" map, the techniques used to treat optimal design problems for divergence type operators can be adapted to solve the problem for the fractional Laplacian as well (see \cite{TT15}). What makes the case of the infinity--Laplacian different is that it does not have a divergence structure, therefore standard tools used in the above mentioned references do not apply, and a different approach is required.

Intuitively, solutions of the problem \eqref{L} may be approximated (in some sense) by solutions of the corresponding problem driven by the $p$-Laplacian. We thus study the problem \eqref{L} through the asymptotic limit as $p\rightarrow\infty$ of minimizers of the problem
\begin{equation}\label{P}\tag{$P_p$}
\textrm{minimize}\,\,J(u):=\frac{1}{p}\displaystyle\int|\nabla u|^p\,dx\quad \textrm{in} \ \mathbb{K}_p,
\end{equation}
where 
$$\mathbb{K}_p = \left\{ u\in W^{1,p}(\R^n)\, | \, u\geq\varphi, \ |\{u>0\}\setminus\Omega| \leq \gamma \right\}, $$
such that (the relations on $\Delta_pu$ being now understood in the distributional sense)

$$
 \left\{
\begin{array}{rcl}
\Delta_p u=0 & \mathrm{in} & \{u>0\}\setminus\Omega, \\ 
\\
\Delta_p u\leq0 & \mathrm{in} & \Omega .
\end{array}
\right.
$$

\medskip

The key to the study of problem \eqref{L} then lies in obtaining uniform in $p$ estimates for \eqref{P}. The strategy is the following: first we study a three parameter family of perturbed problems and prove uniform estimates, and pass to the limit, in two of those parameters, which leads to another perturbed problem - now with just one parameter. We get that any solution of \eqref{P} is a minimizer for this one parameter limiting perturbed problem, for small values of the parameter, and solves an obstacle type problem (with the function $\varphi$ acting as the obstacle). Combining the information on the regularity of solutions of these two problems, we get the (optimal) Lipschitz regularity of solutions of \eqref{P}, as well as regularity results on the two free boundaries (the exterior and the interior). Finally, we obtain uniform in $p$ estimates, which allow us to let $p\rightarrow\infty$ in \eqref{P} and derive conclusions on our original problem ruled by the infinity--Laplace operator.

\bigskip

The paper is organized as follows: in Section \ref{S2}, we introduce a three-parameter penalization functional to study problem \eqref{P} and obtain the existence of minimizers (Proposition \ref{p2.1}). In Section \ref{S3}, we get $C^{1,\alpha}$ estimates for minimizers, uniform in one of the parameters, and pass to the limit according to that parameter to arrive at the study of a two-parameter penalization functional (Corollary \ref{c3.1}). In Section \ref{S4}, we establish the uniform Lipschitz regularity of minimizers (Theorem \ref{t4.1}), which allows us to pass to the limit in one of the remaining two parameters, reducing the problem to the study of the minimizers of a functional now depending on just one parameter (Corollary \ref{c4.2}). The minimizers of this functional, as well as the regularity of the free boundaries, are studied in Section \ref{S5} (Theorems \ref{t5.2} and \ref{t5.3}). The limiting free boundary problem as $p\rightarrow\infty$ is studied in Section \ref{S6}. We first see that, when the remaining parameter is small enough, the minimizers of the one-parameter functional are in fact minimizers of \eqref{P} (Theorem \ref{t6.1}), hence there is no need to pass to the limit as the minimizers of \eqref{P} carry all the properties proved for the minimizers of the parameterized functional (Lipschitz regularity, linear growth away from the free boundary, non-degeneracy); this is the object of Lemma \ref{l6.1}. Moreover, up to a subsequence, solutions of \eqref{P} converge (in some sense) to a function which is a solution of \eqref{L} (Theorem \ref{t6.2}). We finally show that the free boundaries converge in the Hausdorff metric (Theorem \ref{t6.3}).

\section{Preliminaries}\label{S2}
In order to study the problem \eqref{P}, we introduce, for $\sigma,\delta,\varepsilon>0$,  the following three-parameter penalization functional
$$
J_{\sigma,\delta,\varepsilon}(v):=\frac{1}{p}\int|\nabla v|^p\,dx+g_\sigma(v-\varphi)+f_\varepsilon\bigg(\int_{\Omega^c}h_\delta(v)\bigg),
$$
defined for $v\in W^{1,p}(\R^n)$, where
\begin{itemize}
  \item $g_\sigma:\R\rightarrow\R$ is a non-negative decreasing convex function defined as follows:
      $$
      g_\sigma(t):=
      \left\{
      \begin{array}{c}
      -\frac{1}{\sigma}(t-\frac{\sigma}{2}),\,\textrm{ for } t<-\sigma,\\
      \hbox{smooth},\, \textrm{ for }-\sigma\le t<0,\\
      0, \,\textrm{ for }t\geq0;
      \end{array}
      \right.
      $$
  \item $h_\delta:\R\rightarrow\R$ is a piecewise linear function which vanishes on $(-\infty,0]$ and is 1 in $[\delta,+\infty)$;
  \item $f_\varepsilon:\R\rightarrow\R$ is defined as follows:
$$
f_\varepsilon(t):=
\left\{
\begin{array}{c}
\frac{1}{\varepsilon}(t-\gamma),\,\textrm{ for }t\geq\gamma,\\
\varepsilon(t-\gamma),\,\,\textrm{ for }t\leq\gamma.
\end{array}
\right.
$$
\end{itemize}
In other words, $g_\sigma$ penalizes functions which do not lie above $\varphi$ (as is usual in the regularity theory of obstacle-type problems (see \cite{PSU12})); the purpose of function $h_\delta$ is the regularization of $u\mapsto|\{u>0\}\setminus\Omega|$ (as in \cite{CS95}); finally, $f_\varepsilon$ is charging configurations which exceed the given volume of the positivity set (as is done, for example, in \cite{AAC86, BMW06, TT15}).

First, we check that the three-parameter perturbed functional has a minimizer.
\begin{proposition}\label{p2.1}
  The functional $J_{\sigma,\delta,\varepsilon}$ has a minimizer. Moreover, if $u_{\sigma,\delta,\varepsilon}$ is such a minimizer, then
   \begin{equation}\label{2.1}
  0\leq u_{\sigma,\delta,\varepsilon}\leq\|\varphi\|_\infty.
  \end{equation}
\end{proposition}
\begin{proof}
Let $v$ be the minimizer of the energy functional $J(u)$ among the functions that are in $W_0^{1,p}(\Omega)$ and lie above $\varphi$, i.e., $v$ is the unique solution of the obstacle problem for the $p$--Laplace operator with obstacle $\varphi$. We have  $v\geq\varphi$ and $\{v>0\}\setminus\Omega=\emptyset$, and therefore $v\in\mathbb{K}_p$. Since $g_\sigma(v-\varphi)=0$ and $h_\delta(v)=0$ in $\Omega^c$, then
$$
J_{\sigma,\delta,\varepsilon}(v)\leq\frac{1}{p}\int|\nabla v|^p\,dx=:M,
$$
which shows that the functional $J_{\sigma,\delta,\varepsilon}$ is not always infinite and hence there exists a minimizing sequence. This sequence is bounded in $W^{1,p}(\R^n)$ so it has a weakly convergent subsequence and it is standard to check (see, for example, \cite{AC81, BMW06}) that the limit of this subsequence is a minimizer for $J_{\sigma,\delta,\varepsilon}$. For the proof of \eqref{2.1} we refer to \cite{Y16}.
\end{proof}

\section{Passing to the limit as $\sigma\rightarrow0$}\label{S3}
In this section we prove $C^{1,\alpha}$ estimates, uniform in $\sigma$. Regularity theory for elliptic equations provides the $C^{1,\alpha}$ regularity of minimizers but we need to establish estimates independently of the parameter $\sigma$ in order to pass to the limit.
\begin{lemma}\label{l3.1}
  If $u_{\sigma,\delta,\varepsilon}$ is a minimizer of $J_{\sigma,\delta,\varepsilon}$, then
  \begin{equation}\label{3.1}
    \|g'_\sigma(u_{\sigma,\delta,\varepsilon}-\varphi)\|_\infty
    \leq\|\varphi\|_{C^{1,1}}+\frac{1}{\varepsilon\delta}.
  \end{equation}
\end{lemma}
\begin{proof}
  Note that if $u=u_{\sigma,\delta,\varepsilon}$ is a minimizer of $J_{\sigma,\delta,\varepsilon}$, then the Euler-Lagrange equation of the perturbed energy functional takes the form
\begin{equation}\label{3.2}
  \Delta_pu=g'_\sigma(u-\varphi)+f'_\varepsilon\bigg(\int_{\Omega^c}h_\delta(u)\bigg)
  h'_\delta(u)\chi_{\Omega^c},
\end{equation}
where $\chi$ is the characteristic function of $\Omega^c$, the complement of $\Omega$.

Define $\tilde{u}=u-\varphi$. Then \eqref{3.2} in terms of $\tilde{u}$ reads
\begin{equation}\label{3.3}
  \Delta_p(\tilde{u}+\varphi)=g'_\sigma(\tilde{u})+f'_\varepsilon\bigg(\int_{\Omega^c}h_\delta(\tilde{u})\bigg)
  h'_\delta(\tilde{u})\chi_{\Omega^c},
\end{equation}
since $\mathrm{supp}\, \varphi \subset \Omega$, and thus $u=\tilde{u}$ in $\Omega^c$.
For each fixed $\sigma>0$, taking $[g'_\sigma(\tilde{u})]^k$ as a test function in \eqref{3.3}, we get
\begin{eqnarray}\label{3.4}
  && \int|\nabla(\tilde{u}+\varphi)|^{p-2}\nabla(\tilde{u}+\varphi)
  \cdot\nabla\tilde{u}\, . \,  k[g'_\sigma(\tilde{u})]^{k-1}g''_\sigma(\tilde{u})
  +[g'_\sigma(\tilde{u})]^{k+1}\nonumber \\
  && +f'_\varepsilon\bigg(\int_{\Omega^c}h_\delta(\tilde{u})\bigg)
  h'_\delta(\tilde{u})\chi_{\Omega^c}[g'_\sigma(\tilde{u})]^k=0.
\end{eqnarray}
If $k$ is even, then $k[g'_\sigma(\tilde{u})]^{k-1}g''_\sigma(\tilde{u})\leq0$ and $[g'_\sigma(\tilde{u})]^{k+1}\leq0$, because $g_\sigma$ is a decreasing convex function.
Since
$$
(|\xi|^{p-2}\xi-|\eta|^{p-2}\eta)\cdot(\xi-\eta)\geq0,\,\,\,\xi,\eta\in\R^n,
$$
from \eqref{3.4} we get
\begin{eqnarray*}
  && \int|\nabla\varphi|^{p-2}\nabla\varphi
  \cdot\nabla\tilde{u}\, . \,  k[g'_\sigma(\tilde{u})]^{k-1}g''_\sigma(\tilde{u})
  +[g'_\sigma(\tilde{u})]^{k+1}\nonumber \\
  && +f'_\varepsilon\bigg(\int_{\Omega^c}h_\delta(\tilde{u})\bigg)
  h'_\delta(\tilde{u})\chi_{\Omega^c}[g'_\sigma(\tilde{u})]^k\geq0
\end{eqnarray*}
or
$$
  -[g'_\sigma(\tilde{u})]^k\Delta_p\varphi
  +[g'_\sigma(\tilde{u})]^{k+1}
  +f'_\varepsilon\bigg(\int_{\Omega^c}h_\delta(\tilde{u})\bigg)
  h'_\delta(\tilde{u})\chi_{\Omega^c}[g'_\sigma(\tilde{u})]^k\geq0,
$$
which is equivalent to
\begin{equation}\label{3.5}
    -[g'_\sigma(\tilde{u})]^{k+1}\leq
  f'_\varepsilon\bigg(\int_{\Omega^c}h_\delta(\tilde{u})\bigg)
  h'_\delta(\tilde{u})\chi_{\Omega^c}[g'_\sigma(\tilde{u})]^k
  -[g'_\sigma(\tilde{u})]^k\Delta_p\varphi.
\end{equation}
Since $u\geq\varphi$, we have that $g'_\sigma(\tilde{u})$ is supported in $\Omega$. Then \eqref{3.5} leads to
\begin{eqnarray*}
  &&\int_\Omega|g'_\sigma(\tilde{u})|^{k+1}
  \leq\int_\Omega\bigg[f'_\varepsilon\bigg(\int_{\Omega^c}h_\delta(\tilde{u})\bigg)
  h'_\delta(\tilde{u})\chi_{\Omega^c}[g'_\sigma(\tilde{u})]^k
  -[g'_\sigma(\tilde{u})]^k\Delta_p\varphi\bigg]\nonumber \\
  &\leq&\bigg[\int_\Omega\bigg|f'_\varepsilon\bigg(\int_{\Omega^c}h_\delta(\tilde{u})\bigg)
  h'_\delta(\tilde{u})\chi_{\Omega^c}
  +\Delta_p\varphi\bigg|^k\bigg]^{\frac{1}{k}}
  \bigg[\int_\Omega|g'_\sigma(\tilde{u})|^{k+1}\bigg]^{\frac{k}{k+1}},
\end{eqnarray*}
therefore
$$
\|g'_\sigma(\tilde{u})\|_{L^{k+1}(\Omega)}
\leq\big(\|\varphi\|_{C^{1,1}}
+\frac{1}{\varepsilon\delta}\big)|\Omega|^{\frac{1}{k}}.
$$
Letting $k\rightarrow+\infty$, we get \eqref{3.1}.
\end{proof}
Observe that Lemma \ref{l3.1} gives the boundedness of the right-hand side of \eqref{3.2}, independently of $\sigma$. This provides $C^{1,\alpha}$ estimates, uniform in $\sigma$ (see \cite{D83} and \cite{T84}).
\begin{theorem}\label{t3.1}
  If $u_{\sigma,\delta,\varepsilon}$ is a minimizer of $J_{\sigma,\delta,\varepsilon}$, then, for any compact set $K\in\R^n$ and for any $\alpha\in(0,1)$, there holds
  $$
  \|u_{\sigma,\delta,\varepsilon}\|_{C^{1,\alpha}(K)}\leq C\big(\|\varphi\|_{C^{1,1}}+\frac{1}{\varepsilon\delta}\big),
  $$
  where $C=C(K,\alpha,n,p)$ is a positive constant.
\end{theorem}
\begin{corollary}\label{c3.1}
 Up to a subsequence $\sigma\rightarrow0$, the function $u_{\sigma,\delta,\varepsilon}$ converges to a function $u_{\delta,\varepsilon}$, weakly in $W^{1,p}(\R^n)$ and locally uniformly in $C^{1,\alpha}(\R^n)$. Moreover, $u_{\delta,\varepsilon}\geq\varphi$.
\end{corollary}
\begin{proof}
The first part is a direct consequence of Theorem \ref{t3.1} and the Arzel\`a-Ascoli theorem. To see  that the limit $u_{\delta,\varepsilon}$ lies above $\varphi$, let $\epsilon>0$ and $K\subset\R^n$ be a compact set. Then, for small $\sigma>0$, we have the inclusion
  $$
  \{u_{\delta,\varepsilon}-\varphi<-\epsilon\}\subset\{u_{\sigma,\delta,\varepsilon}-\varphi<-\epsilon/2\}
  \cap K.
  $$
  Hence
  $$
  +\infty>M\geq\int g_\sigma(u_{\sigma,\delta,\varepsilon}-\varphi)\,dx\geq\frac{\epsilon}{2\sigma}
  |\{u_{\sigma,\delta,\varepsilon}-\varphi<-\epsilon/2\}\cap K|,
  $$
and  therefore $|\{u_{\sigma,\delta,\varepsilon}-\varphi<-\epsilon/2\}\cap K|=0$, since otherwise we get a contradiction if $\sigma>0$ is small enough.

\end{proof}

\section{Uniform Lipschitz regularity}\label{S4}

In the previous section, we were able to pass to the limit as $\sigma\rightarrow0$ to get a function $u_{\delta,\varepsilon}$ which still depends on the parameters $\delta$ and $\varepsilon$. One would expect the limit $u_{\delta,\varepsilon}$ from Corollary \ref{c3.1} to be a minimizer of the functional
$$
\frac{1}{p}\int|\nabla v|^p+f_\varepsilon\bigg(\int_{\Omega^c}h_\delta(v)\bigg)
$$
over functions which lie above $\varphi$. But this is not necessarily true since $u\mapsto f_\varepsilon\big(\int_{\Omega^c}h_\delta(u)\big)$ is not convex. However, we are able to show the uniform (in $\delta$) Lipschitz continuity of $u_{\delta,\varepsilon}$, which then allows one to let $\delta\rightarrow0$. We begin with an auxiliary lemma.
\begin{lemma}\label{l4.1}
For all $v\in W^{1,p}(\R^n)$, with $v\geq\varphi$, we have
  $$
  \int|\nabla v|^{p-2}\nabla v\cdot\nabla(v-u_{\delta,\varepsilon})+
  f_\varepsilon'\bigg(\int_{\Omega^c}h_\delta(u_{\delta,\varepsilon})\bigg)
  \int_{\Omega^c}h_\delta'(u_{\delta,\varepsilon})(v-u_{\delta,\varepsilon})\geq0.
  $$
  
\end{lemma}
\begin{proof}
  Since $u_{\sigma,\delta,\varepsilon}$ is a minimizer for $J_{\sigma,\delta,\varepsilon}$, then the functional
  $$
  F(t):=J_{\sigma,\delta,\varepsilon}(u_{\sigma,\delta,\varepsilon}+
  t(v-u_{\sigma,\delta,\varepsilon})),
  $$
  with $v\in W^{1,p}(\R^n)$ and $t \geq 0$, takes a minimum at $t=0$, which implies that $F'(0)\geq0$, i.e.,
  \begin{eqnarray}\label{4.1}
    &&\int|\nabla u_{\sigma,\delta,\varepsilon}|^{p-2}\nabla u_{\sigma,\delta,\varepsilon}\cdot\nabla(v-u_{\sigma,\delta,\varepsilon})
    +g_\sigma'(u_{\sigma,\delta,\varepsilon}-\varphi)(v-u_{\sigma,\delta,\varepsilon}) \nonumber \\
    &+&f_\varepsilon'\bigg(\int_{\Omega^c}h_\delta(u_{\sigma,\delta,\varepsilon})\bigg)
    \int_{\Omega^c}h_\delta'(u_{\sigma,\delta,\varepsilon})(v-u_{\sigma,\delta,\varepsilon})\geq0.
  \end{eqnarray}
  On the other hand, monotonicity of the $p$-Dirichlet energy and the function $g_\sigma'$ gives
  \begin{eqnarray*}
   && \int(|\nabla v|^{p-2}\nabla v-|\nabla u_{\sigma,\delta,\varepsilon}|^{p-2}\nabla u_{\sigma,\delta,\varepsilon})\cdot\nabla(v-u_{\sigma,\delta,\varepsilon})\nonumber\\
    &+&(g_\sigma'(v-\varphi)-g_\sigma'(u_{\sigma,\delta,\varepsilon}-\varphi))(v-u_{\sigma,\delta,\varepsilon})\geq0,
  \end{eqnarray*}
  which, together with \eqref{4.1}, provides
  \begin{eqnarray*}
    &&\int|\nabla v|^{p-2}\nabla v\cdot\nabla(v-u_{\sigma,\delta,\varepsilon})+g_\sigma'(v-\varphi)(v-u_{\sigma,\delta,\varepsilon}) \nonumber \\
    &+&f_\varepsilon'\bigg(\int_{\Omega^c}h_\delta(u_{\sigma,\delta,\varepsilon})\bigg)
    \int_{\Omega^c}h_\delta'(u_{\sigma,\delta,\varepsilon})(v-u_{\sigma,\delta,\varepsilon})\geq0,\,\,\,\forall v\in W^{1,p}(\R^n).
  \end{eqnarray*}
  In particular, for $v\geq\varphi$, we have
  \begin{eqnarray}\label{4.2}
     &&\int|\nabla v|^{p-2}\nabla v\cdot\nabla(v-u_{\sigma,\delta,\varepsilon})\\
   &+&f_\varepsilon'\bigg(\int_{\Omega^c}h_\delta(u_{\sigma,\delta,\varepsilon})\bigg)
    \int_{\Omega^c}h_\delta'(u_{\sigma,\delta,\varepsilon})(v-u_{\sigma,\delta,\varepsilon})\geq0.\nonumber
  \end{eqnarray}
  Our aim now is to show that we can pass to the limit in \eqref{4.2}, as $\sigma\rightarrow0$, to conclude the proof. Note that the weak convergence $u_{\sigma,\delta,\varepsilon}\rightharpoonup u_{\delta,\varepsilon}$ in $W^{1,p}$, as $\sigma\rightarrow0$, allows one to pass to the limit in the first term of \eqref{4.2}.  Since $f_\varepsilon$ and $h_\delta$ are smooth, we do not have any problems passing to the limit in the first part of the second term of \eqref{4.2}. As for the second part, the proof is the same as in \cite{Y16}. We bring it here for the sake of completeness. Observe that
  \begin{eqnarray*}
    &&\bigg|\int_{\Omega^c}h_\delta'(u_{\sigma,\delta,\varepsilon})(v-u_{\sigma,\delta,\varepsilon})-\int_{\Omega^c}h_\delta'(u_{\delta,\varepsilon})(v-u_{\delta,\varepsilon})\bigg|\nonumber\\
    &\leq& \bigg|\int_{\Omega^c}h_\delta'(u_{\sigma,\delta,\varepsilon})(u_{\sigma,\delta,\varepsilon}-v)\bigg|
    +\bigg|\int_{\Omega^c}(h_\delta'(u_{\sigma,\delta,\varepsilon})-h_\delta'(u_{\delta,\varepsilon}))(v-u_{\delta,\varepsilon})\bigg|\nonumber \\
    &\leq& C(\delta)\|u_{\sigma,\delta,\varepsilon}-u_{\delta,\varepsilon}\|_{L^2}+\bigg|\int_{\Omega^c}(h_\delta'(u_{\sigma,\delta,\varepsilon})-h_\delta'(u_{\delta,\varepsilon}))(v-u_{\delta,\varepsilon})\bigg|\nonumber \\
    &=& o(1)+\bigg|\int_{\Omega^c}(h_\delta'(u_{\sigma,\delta,\varepsilon})-h_\delta'(u_{\delta,\varepsilon}))(v-u_{\delta,\varepsilon})\bigg|.
  \end{eqnarray*}
  Hence, if
  \begin{equation}\label{4.3}
    \bigg|\int_{\Omega^c}(h_\delta'(u_{\sigma,\delta,\varepsilon})-h_\delta'(u_{\delta,\varepsilon}))(v-u_{\delta,\varepsilon})\bigg|\rightarrow0,
  \end{equation}
  as $\sigma\rightarrow0$, then the proof will be finished. Note that $(h_\delta'(u_{\sigma,\delta,\varepsilon})-h_\delta'(u_{\delta,\varepsilon}))$ is bounded, and $(u-v)\in L^2$. Therefore, for a given $\epsilon>0$, there exists a large enough $r$ guaranteeing 
  $$
  \bigg|\int_{\Omega^c}(h_\delta'(u_{\sigma,\delta,\varepsilon})-h_\delta'(u_{\delta,\varepsilon}))(v-u_{\delta,\varepsilon})
  -\int_{\Omega^c\cap B_r}(h_\delta'(u_{\sigma,\delta,\varepsilon})-h_\delta'(u_{\delta,\varepsilon}))(v-u_{\delta,\varepsilon})\bigg|
  \leq\epsilon,
  $$
  where $B_r$ is the ball of radius $r$. Applying the dominated convergence theorem on the compact set $B_r$, one concludes
  $$
  \int_{\Omega^c\cap B_r}(h_\delta'(u_{\sigma,\delta,\varepsilon})-h_\delta'(u_{\delta,\varepsilon}))(v-u_{\delta,\varepsilon})\rightarrow0.
  $$
  Thus, \eqref{4.3} is true, which means that we can pass to the limit in \eqref{4.2} to conclude the proof of the lemma.
\end{proof}
\begin{corollary}\label{c4.1}
The function $u_{\delta,\varepsilon}$ satisfies
$$
    \Delta_p u_{\delta,\varepsilon}=f_\varepsilon'\bigg(\int_{\Omega^c}h_\delta(u_{\delta,\varepsilon})\bigg)h_\delta'(u_{\delta,\varepsilon})\chi_{\Omega^c}
    \,\,\textrm{ in }\,\,\{u_{\delta,\varepsilon}>\varphi\},
$$
  and
  \begin{equation}\label{4.5}
    -\|\varphi\|_{C^{1,1}}-\frac{1}{\varepsilon\delta}\leq\Delta_p u_{\delta,\varepsilon}\leq f_\varepsilon'\bigg(\int_{\Omega^c}h_\delta(u_{\delta,\varepsilon})\bigg)h_\delta'(u_{\delta,\varepsilon})\chi_{\Omega^c}
    \,\,\textrm{ in }\,\,\R^n.
  \end{equation}
\end{corollary}
\begin{proof}
    It remains to check the lower bound of \eqref{4.5}, since the rest is a consequence of Lemma \ref{l4.1}. The weak convergence of $u_{\sigma,\delta,\varepsilon}\rightharpoonup u_{\delta,\varepsilon}$ in $W^{1,p}$ and the uniform bound on the right hand side of \eqref{3.2} give the lower bound in \eqref{4.5}.
  \end{proof}

In order to pass to the limit as $\delta\rightarrow0$, we need to prove estimates which are uniform in $\delta$. The proof of the following theorem is from \cite{Y16}, with small adaptations.
\begin{theorem}[Optimal regularity]\label{t4.1}
There exists a constant $C>0$, depending only on $n$ and $p$, such that
  $$
  |\nabla u_{\delta,\varepsilon}|\leq C(\|\varphi\|_{C^1}+\delta\|\varphi\|_{C^{1,1}}+
  \frac{1}{\varepsilon}).
  $$
\end{theorem}
\begin{proof}
  We divide the proof into three steps.\\

  \textsc{Step 1.} Let $x_0\in\{u_{\delta,\varepsilon}\leq\delta\}$ and let $w(y):=\frac{1}{\delta}u_{\delta,\varepsilon}(x_0+\delta y)$. Then $w(0)\leq1$ and, from \eqref{4.5}, we also have
  $$
  -\delta\|\varphi\|_{C^{1,1}}-\frac{1}{\varepsilon}\leq\Delta_p w\leq\frac{1}{\varepsilon}.
  $$
  Thus, $w$ is a non-negative function with bounded $p$-Laplacian and then 
  $$|w|\leq C(n,p)(\delta\|\varphi\|_{C^{1,1}}+\frac{1}{\varepsilon}+1)\quad \mathrm{in} \ B_1, $$ 
  for a constant $C(n,p)>0$. Hence, the bound holds for $|\nabla w(0)|$ (see, for example, \cite{U68}), and we conclude by noting that  $\nabla w(0)=\nabla u_{\delta,\varepsilon}(x_0)$.\\

  \textsc{Step 2.} Let $x_0\in\{u_{\delta,\varepsilon}=\varphi\}$. For $w(y):=\frac{1}{\delta}(u_{\delta,\varepsilon}(x_0+\delta y)-\varphi(x_0+\delta y))$ one has $w\geq0$, $w(0)=0$ and
  $$
  -2\delta\|\varphi\|_{C^{1,1}}-\frac{1}{\varepsilon}\leq
  \Delta_p w\leq\delta\|\varphi\|_{C^{1,1}}+\frac{1}{\varepsilon}.
  $$
  Observe that $\nabla u_{\delta,\varepsilon}(x_0)=\nabla w(0)+\nabla\varphi(x_0)$, and therefore, as in the Step 1, the elliptic regularity theory provides
  $$
  |\nabla u_{\delta,\varepsilon}(x_0)|\leq C(n,p)(\|\varphi\|_{C^1}+\delta\|\varphi\|_{C^{1,1}}+\frac{1}{\varepsilon}+1).
  $$

  \textsc{Step 3.} Let now $x_0\in\{u_{\delta,\varepsilon}>\varphi\}\cap\{u_{\delta,\varepsilon}>\delta\}$. Define
  $$
  d=\dist(x_0,\partial(\{u_{\delta,\varepsilon}>\varphi\}\cap\{u_{\delta,\varepsilon}>\delta\})),
  $$
  and let $y_0$ be the point where the distance is attained, i.e., $|x_0-y_0|=d$. In particular, one has $u_{\delta,\varepsilon}(y_0)=\delta$ or $u_{\delta,\varepsilon}(y_0)=\varphi(y_0)$.

  In case $u_{\delta,\varepsilon}(y_0)=\delta$, the function $w(y):=\frac{1}{d}(u_{\delta,\varepsilon}(x_0+dy)-\delta)$ is a non-negative $p$-harmonic function in $B_1$, and, for a point $\tilde{y}_0\in\partial B_1$ corresponding to $y_0$, one has $w(\tilde{y}_0)=0$. Hence, Step 1 can be applied to get
  \begin{equation}\label{4.6}
    |\nabla w(\tilde{y}_0)|\leq C(n,p)(\delta\|\varphi\|_{C^{1,1}}+\frac{1}{\varepsilon}+1).
  \end{equation}
  On the other hand, for a constant $c(n,p)>0$, the  Harnack inequality provides the lower bound
  \begin{equation}\label{4.7}
    w(y)\geq c(n,p)w(0)\,\,\textrm{ in }\,\,B_{1/2}.
  \end{equation}
  Now let $v$ be the function which is $p$-harmonic in the ring $B_1\setminus B_{1/2}$, vanishes on $\partial B_1$ and equals $c(n,p)w(0)$ along $\partial B_{1/2}$. By the comparison principle, one has $v\leq w$ in $B_1\setminus B_{1/2}$. But since $w(\tilde{y}_0)=v(\tilde{y}_0)$, 
  $$\nabla w(\tilde{y}_0)\cdot\nu(\tilde{y}_0)\geq\nabla v(\tilde{y}_0)\cdot\nu(\tilde{y}_0),$$
  where $\nu$ is the inner normal vector to $B_1$. Hence, combining this with \eqref{4.6} and \eqref{4.7}, we arrive at
  $$
  w(y)\leq C(n,p)(\delta\|\varphi\|_{C^{1,1}}+\frac{1}{\varepsilon}+1)\,\,\textrm{ in }\,\,B_{1/2},
  $$
  and again, as in the previous cases, the elliptic regularity theory gives
  $$
  |\nabla u_{\delta,\varepsilon}(x_0)|=|\nabla w(0)|\leq C(n,p)(\delta\|\varphi\|_{C^{1,1}}+\frac{1}{\varepsilon}+1).
  $$

  If $u_{\delta,\varepsilon}(y_0)=\varphi(y_0)$, then $w(y):=\frac{1}{d}(u_{\delta,\varepsilon}(x_0+dy)-\varphi(x_0+dy))$ is a non-negative $p$-harmonic function in $B_1$ which vanishes at a point on $\partial B_1$, where one has the estimate from Step 2. Then, as above, a barrier argument can be used to conclude the proof. 
 
Since the gradient of $u_{\delta,\varepsilon}$ has a jump along the free boundary $\partial\{u_{\delta,\varepsilon}>0\}$, the Lipschitz regularity is optimal (see \cite{AC81}).
\end{proof}
As a consequence of Theorem \ref{t4.1} and the Arzel\`a-Ascoli theorem we obtain the next result. Observe that we can fix $p>n$, bound the $W^{1,p}$-norm by energy considerations and directly obtain the H\"older continuity by embedding (see \cite{AT, LQT, MRU09}).

\begin{corollary}\label{c4.2}
  If $u_{\sigma,\delta,\varepsilon}$ is a minimizer of $J_{\sigma,\delta,\varepsilon}$, then $u_{\sigma,\delta,\varepsilon}$ converges weakly (up to a subsequence as $\sigma,\delta\rightarrow0$) in $W^{1,p}$ to a function $u_\varepsilon$. This convergence is locally uniform in $C^{\alpha}$, for any $\alpha\in(0,1)$. Moreover, there exists a constant $C=C(n,p)>0$ such that
  $$
  |\nabla u_\varepsilon|\leq C(\|\varphi\|_{C^1}+\frac{1}{\varepsilon}).
  $$
\end{corollary}

\section{Regularity of the free boundaries}\label{S5}
In this section we see that the function $u_\varepsilon$ from Corollary \ref{c4.2} is a minimizer for a certain functional. This yields information on the regularity of the two free boundaries of the problem, namely, the interior free boundary
$$\partial \left( \{ u_\varepsilon > \varphi \} \cap \Omega \right)$$
and the exterior free boundary
$$\partial \left( \{ u_\varepsilon > 0 \}  \right).$$

\begin{theorem}\label{t5.1}
The function $u_\varepsilon$ is a local minimizer of
  $$
  J_\varepsilon(u):=\frac{1}{p}\int|\nabla u|^p+
  f_\varepsilon(|\{u>0\}\setminus\Omega|)
  $$
  over the functions in $W^{1,p} (\R^n)$ which lie above $\varphi$.
\end{theorem}
\begin{proof}
  Suppose the assertion of the theorem fails. Then there is $\epsilon>0$ and $v\geq\varphi$, with $v-u_\varepsilon$ supported in a ball $B_r(x_0)$, such that
  \begin{eqnarray}\label{5.1}
    &&\int_{B_r(x_0)}\frac{|\nabla v|^p}{p}+f_\varepsilon\bigg(\int_{\Omega^c\cap B_r(x_0)}\chi_{\{v>0\}}\bigg)\nonumber \\
    &<&\int_{B_r(x_0)}\frac{|\nabla u_\varepsilon|^p}{p}+f_\varepsilon\bigg(\int_{\Omega^c\cap B_r(x_0)}\chi_{\{u_\varepsilon>0\}}\bigg)-\epsilon.
  \end{eqnarray}
  Since $h_\delta(v)\rightarrow\chi_{\{v>0\}}$ as $\delta\rightarrow0$, we have
  \begin{eqnarray}
    &&\int_{B_r(x_0)}\frac{|\nabla v|^p}{p}+f_\varepsilon\bigg(\int_{\Omega^c\cap B_r(x_0)}\chi_{\{v>0\}}\bigg)\nonumber\\
    &=&\lim_{\delta\rightarrow0}
    \int_{B_r(x_0)}\frac{|\nabla v|^p}{p}+f_\varepsilon\bigg(\int_{\Omega^c\cap B_r(x_0)}h_\delta(v)\bigg)\\
    &=&\lim_{\sigma,\delta\rightarrow0}\int_{B_r(x_0)}\frac{|\nabla v|^p}{p}+g_\sigma(v-\varphi)+f_\varepsilon\bigg(\int_{\Omega^c\cap B_r(x_0)}h_\delta(v)\bigg)\nonumber.
  \end{eqnarray}
  On the other hand, for a fixed small $\tau>0$, using Fatou's lemma and the fact that $h_\delta(u_\delta)=\chi_{\{u_\delta>0\}}$ on $\{u_\varepsilon>\tau\}$ for $\delta$ small enough, we get
  \begin{eqnarray}\label{5.3}
    &&\int_{B_r(x_0)}\frac{|\nabla u_\varepsilon|^p}{p}+f_\varepsilon\bigg(\int_{\Omega^c\cap B_r(x_0)}\chi_{\{u_\varepsilon>0\}}\bigg)-\epsilon\nonumber\\
    &\leq&\int_{B_r(x_0)}\frac{|\nabla u_\varepsilon|^p}{p}
    +f_\varepsilon\bigg(\int_{\Omega^c\cap B_r(x_0)\cap\{u_\varepsilon\geq\tau\}}\chi_{\{u_\varepsilon>0\}}\bigg)-\frac{\epsilon}{2}\nonumber\\
    &\leq&\liminf_{\delta\rightarrow0}\int_{B_r(x_0)}\frac{|\nabla u_{\delta,\varepsilon}|^p}{p}
    +f_\varepsilon\bigg(\int_{\Omega^c\cap B_r(x_0)\cap\{u_\varepsilon\geq\tau\}}\chi_{\{u_{\delta,\varepsilon}>0\}}\bigg)-\frac{\epsilon}{2}\nonumber\\
    &\leq&\liminf_{\delta\rightarrow0}\int_{B_r(x_0)}\frac{|\nabla u_{\delta,\varepsilon}|^p}{p}
    +f_\varepsilon\bigg(\int_{\Omega^c\cap B_r(x_0)\cap\{u_\varepsilon\geq\tau\}}h_\delta(u_{\delta,\varepsilon})\bigg)-\frac{\epsilon}{2}\nonumber\\
    &\leq&\liminf_{\sigma,\delta\rightarrow0}\int_{B_r(x_0)}\frac{|\nabla u_{\sigma,\delta,\varepsilon}|^p}{p}
    +f_\varepsilon\bigg(\int_{\Omega^c\cap B_r(x_0)\cap\{u_\varepsilon\geq\tau\}}h_\delta(u_{\sigma,\delta,\varepsilon})\bigg)-\frac{\epsilon}{2}\nonumber\\
    &\leq&\liminf_{\sigma,\delta\rightarrow0}\int_{B_r(x_0)}\frac{|\nabla u_{\sigma,\delta,\varepsilon}|^p}{p}+g_\sigma(u_{\sigma,\delta,\varepsilon}-\varphi)\nonumber\\
    &+&f_\varepsilon\bigg(\int_{\Omega^c\cap B_r(x_0)}h_\delta(u_{\sigma,\delta,\varepsilon})\bigg)-\frac{\epsilon}{2}.
  \end{eqnarray}
  Combining \eqref{5.1}-\eqref{5.3}, we get that
  $$
  J_{\sigma,\delta,\varepsilon}(v)<J_{\sigma,\delta,\varepsilon}(u_{\sigma,\delta,\varepsilon}) -\frac{\epsilon}{4},
  $$
  which contradicts the minimality of $u_{\sigma,\delta,\varepsilon}$ once $\sigma,\delta$ are small enough.
\end{proof}

The following proposition contains what can be interpreted as the Euler-Lagrange equation satisfied by $u_\varepsilon$ and is a direct consequence of \eqref{3.2}.
\begin{proposition}\label{c5.1}
 The function $u_\varepsilon$ satisfies 
 $$
\left\{
\begin{array}{rcl}
 \Delta_p u_\varepsilon \leq 0 & \mathrm{in} & \Omega,  \\
 \Delta_p u_\varepsilon=0 & \mathrm{in} & \Omega\cap\{u_\varepsilon>\varphi\},\\
 \Delta_p u_\varepsilon\geq0 & \mathrm{in} & \Omega^c,\\
 \Delta_p u_\varepsilon=0 & \mathrm{in} & \{u_\varepsilon>0\}\setminus\Omega. 
\end{array}
\right.
$$
It thus minimizes the $p$-Dirichlet energy over the set of functions in $W^{1,p}(\Omega)$ which lie above $\varphi$, i.e., $u_\varepsilon$ is a solution of the obstacle problem for the $p$-Laplacian in $\Omega$, with $\varphi$ acting as the obstacle and having $u_\varepsilon|_{\partial\Omega}$ as boundary data.
\end{proposition}
The regularity theory for the obstacle problem (see \cite{PSU12}) now yields the regularity of the interior free boundary.
\begin{theorem}\label{t5.2}
  If $\Delta_p\varphi$ is uniformly negative in $\{\varphi>0\}$, then the interior free boundary is smooth in the sense that $\mathcal{H}^{n-1}(\partial(\{u_\varepsilon>\varphi\}\cap\Omega))<\infty$, where $\mathcal{H}^{n-1}$ is the $(n-1)$-dimensional Hausdorff measure of the set.
\end{theorem}

Observe that as long as we do not touch the interior contact set, all tools from the theory of $p$-harmonic functions with linear growth are available, and so we can use the corresponding result of \cite{BMW06} to obtain the following representation theorem, which, in particular, gives information on the regularity of the exterior free boundary.
\begin{theorem}\label{t5.3}
For any compact set $K$ we have
  \begin{itemize}
    \item[(1)]
    $$
    \mathcal{H}^{n-1}(K\cap\partial\{u_\varepsilon>0\})<\infty;
    $$
    \item[(2)] There exists a Borel measure $q_{u_\varepsilon}$ such that
        $$
        \Delta_p u_\varepsilon|_{(\Omega\cap\{u_\varepsilon=\varphi\})^c}=
        q_{u_\varepsilon}\mathcal{H}^{n-1}|_{\partial\{u_\varepsilon>0\}};
        $$
    \item[(3)] There exist positive constants $c$ and $C$, depending only on $n$, $\|\varphi\|_{C^1}$, $K$ and $\varepsilon$, such that, for $x_0\in\partial\{u_\varepsilon>0\}$ and $B_r(x_0)\subset(\Omega\cap\{u_\varepsilon=\varphi\})^c$, there holds
        $$
        c\leq q_{u_\varepsilon}(x_0)\leq C;
        $$
    \item[(4)]
    $$
    cr^{n-1}\leq\mathcal{H}^{n-1}(B_r(x_0)\cap\partial\{u_\varepsilon>0\})\leq Cr^{n-1};
    $$
    \item[(5)] For $\mathcal{H}^{n-1}$-almost every $x_0\in\{u_\varepsilon>0\}$, one has
        $$
        u_\varepsilon(x_0+x)=q_{u_\varepsilon}(x_0)\max\{-x\cdot\nu(x_0),0\}+o(|x|),
        $$
        where $\nu(x_0)$ is the outer normal vector to the reduced boundary of $\{u_\varepsilon>0\}$ at $x_0$;
    \item[(6)] $q_{u_\varepsilon}$ is constant $\mathcal{H}^{n-1}$-almost everywhere on $\partial\{u_\varepsilon>0\}$;
    \item[(7)] The exterior free boundary $\partial\{u_\varepsilon>0\}$ is smooth except on a $\mathcal{H}^{n-1}$-null set.
  \end{itemize}
\end{theorem}

\section{The limiting problem as $p\rightarrow\infty$}\label{S6}
In this section, we study the limit, as $p\rightarrow\infty$, of problem \eqref{P} and, for that purpose, we need to find bounds for its solutions which are independent of $p$.  Observe that, from Theorem \ref{t4.1}, we know that local minimizers of $J_\varepsilon(u)$ are Lipschitz continuous, i.e.,
$$
\|\nabla u_\varepsilon\|_\infty\leq C,
$$
where $C>0$ is a constant depending on $p$. The following lemma reveals the nature of that dependence. 
\begin{lemma}\label{l6.1}
There is a constant $L>0$, independent of $p$, such that
$$
\|\nabla u_\varepsilon\|_\infty\leq L\varepsilon^{-1/p}.
$$
In addition, for a constant $\theta>0$, independent of $p$,
  $$
  u_\varepsilon(x)>\varepsilon^{1/p} \, \theta \, \dist(x,\partial\{u_\varepsilon>0\})
  $$
  and  the strong non-degeneracy property holds, i.e.,
  $$
  \sup_{B_r(x_0)}u_\varepsilon\geq\varepsilon^{1/p}\, \theta \, r,
  $$
  where $x_0\in \partial\{u_\varepsilon>0\}$.
\end{lemma}
\begin{proof}
  From Theorem \ref{t5.1}, we know that $u_\varepsilon$ is a local minimizer of $J_\varepsilon$ over the set of functions which lie above $\varphi$. To prove the desired Lipschitz estimate for such minimizers, we follow the approach of \cite{AC81} (see also \cite{RT12}). We may assume that $p$ is large because we are interested in the limiting problem. Let $h$ be the $p$-harmonic replacement of $u_\varepsilon$ in the ball $B_r(x_0)$, i.e., $h$ is $p$-harmonic in the ball and agrees with $u_\varepsilon$ on the boundary of the ball. The minimizer $u_\varepsilon$ stays above $\varphi$, and thanks to the maximum principle, so does $h$. Hence, $h$ is an admissible function, and since $u_\varepsilon$ is a local minimizer of $J_\varepsilon$, then over $B_r(x_0)$ one has that $J_\varepsilon(u_\varepsilon)\leq J_\varepsilon(h)$. The latter implies
  \begin{equation}\label{6.1}
  c_0\int_{B_r(x_0)}|\nabla(u_\varepsilon-h)(x)|^p\,dx\leq\frac{1}{\varepsilon}|\{u_\varepsilon=0\}\cap B_r(x_0)|,
  \end{equation}
  where $c_0>0$ is a constant depending only on the dimension. 
  
  Next, for any direction $\nu$, define
  $$
  s_\nu:=\min \left\{ s\,|\,\frac{1}{4}\leq s\leq1;\,u_\varepsilon(x_0+rs\nu)=0 \right\}
  $$
  if such a set is nonempty; otherwise, assume $s_\nu=1$. Then
\begin{eqnarray}\label{6.2}
  h(x_0+rs_\nu\nu) &=& \int_{s_\nu}^{1}\frac{d}{ds}(u_\varepsilon-h)(x_0+rs\nu)\,ds \\
  &\leq& r(1-s_\nu)^{1/p'}\bigg[\int_{s_\nu}^{1}|\nabla(h-u_\varepsilon)(x_0+s\nu)|^p\,ds\bigg]^{1/p},\nonumber
\end{eqnarray}
where $\frac{1}{p}+\frac{1}{p'}=1$. On the other hand, by the Harnack inequality, one has that
$$
\inf_{B_{2r/3}(x_0)}h\geq c_1h(x_0),
$$
with a constant $c_1>0$ depending only on the dimension (see \cite{KMV96}). 

In order to proceed, we construct a barrier function $v$ such that
$$
\left\{
  \begin{array}{ll}
    \Delta_pv=0 & \hbox{in }\,B_1(0)\setminus B_{2/3}(0), \\
    v=0 & \hbox{on }\,\partial B_1(0), \\
    v=c_1 & \hbox{in }\,B_{2/3}(0),
  \end{array}
\right.
$$
with the same constant $c_1$ appearing in the Harnack inequality. By the Hopf maximum principle, there exists a constant $c_2>0$, depending only on dimension, such that
\begin{equation}\label{6.3}
  v(x)\geq c_2(1-|x|).
\end{equation}
Then the Harnack inequality, the maximum principle and \eqref{6.3} lead to
$$
h(x_0+rx)\geq h(x_0)v(x)\geq c_2h(x_0)(1-|x|),
$$
which, together with \eqref{6.2}, provides
$$
r^p\bigg[\int_{s_\nu}^{1}|\nabla(h-u_\varepsilon)(x_0+s\nu)|^p\,ds\bigg]\geq c_2^ph^p(x_0)(1-s_\nu).
$$
Integrating the latter with respect to $\nu$ over $\mathbb{S}^{n-1}$, we arrive at
$$
  \bigg(\frac{c_2h(x)}{r}\bigg)^p\int_{B_r(x)\setminus B_{r/4}(x)}\chi_{\{u_\varepsilon=0\}\cap B_r(x_0)}\,dx\leq c_3\int_{B_{r}(x)}|\nabla(h-u_\varepsilon)(x)|^p\,dx,
$$
with a constant $c_3>0$ independent of $p$. Therefore,
\begin{equation}\label{6.4}
   \bigg(\frac{c_2h(x)}{r}\bigg)^p|\{u_\varepsilon=0\}\cap B_r(x_0)|\leq c_3\int_{B_{r}(x)}|\nabla(h-u_\varepsilon)(x)|^p\,dx.
\end{equation}
Define $\rho:=\dist(x,\partial\{u_\varepsilon>0\})$. For each $\xi\in(0,1)$, let $h_\xi$ be the $p$-harmonic replacement of $u_\varepsilon$ in $B_{\rho+\xi}(x)$. Then \eqref{6.1} and \eqref{6.4}, together with standard elliptic estimates, lead  to
$$
u_\varepsilon(x)=h_\xi(x)+o(1)\leq L\frac{1}{\varepsilon^p}(\rho+\xi)+o(1),
$$
with a constant $L>0$ depending only on the dimension and $\gamma$. Letting $\xi\rightarrow0$ in the last inequality, we get
$$
u_\varepsilon(x)\leq L\varepsilon^{-1/p}\rho,
$$
which implies that $u_\varepsilon$ is Lipschitz continuous, with a Lipschitz constant not exceeding $L\varepsilon^{-1/p}$. The first part of the lemma is proved.

\medskip

The proof of the (strong) non-degeneracy result is classical in variational free boundary theory. It is based on cutting a small hole around the free boundary point and comparing the result with the original optimal configuration while keeping track of the precise constants which appear on the estimates. For the details, we refer the reader to Theorem 6.2 of \cite{T10}.
\end{proof}

As is usual in this type of problems (see, for example, \cite{AAC86, ACS87, AC81, BMW06, OT06, RT12, T10, TT15}) it turns out that, for $\varepsilon>0$ small enough, the function $u_\varepsilon$ from Corollary \ref{c4.2} reaches the desired volume, i.e., it becomes a solution of \eqref{P} carrying all the properties proved above. This means that we do not need to pass to the limit as $\varepsilon\rightarrow0$.

\begin{theorem}\label{t6.1}
  For $\varepsilon>0$ small enough, the function $u_\varepsilon$ solves the problem \eqref{P}.
\end{theorem}
\begin{proof}
  From Theorem \ref{t5.1} we know that $u_\varepsilon$ minimizes $J_\varepsilon$. This implies that when a fixed $\varepsilon>0$ is small enough, we have $|\{u_\varepsilon>0\}\setminus\Omega|=\gamma$ (Theorem 3.1 of \cite{BMW06}). Then $f_\varepsilon(|\{u_\varepsilon>0\}\setminus\Omega|)=0$ and, recalling Proposition \ref{c5.1}, one concludes that $u_\varepsilon\in\mathbb{K}_p$ solves \eqref{P}.
\end{proof}

From now on, $\varepsilon \ll 1$ is fixed and, to emphasize  the $p$ dependence, we denote with $u_p$ a solution of problem \eqref{P}.

\begin{corollary}\label{c6.1}
  If $u_p$ is a solution of \eqref{P}, then there is a constant $c>0$, independent of $p$, such that
  $$
  \|\nabla u_p\|_\infty\leq c.
  $$
  Moreover, $u_p$ grows linearly away from the free boundary, uniformly in $p$, i.e., there is a constant $\theta>0$, independent of $p$, such that
  $$
  u_p(x)>\theta \, \dist(x,\partial\{u_p>0\}),\,\,\forall x\in\{u_p>0\}.
  $$
  Finally, $u_p$ is strongly non-degenerate, that is, for any $x_0\in\partial\{u_p>0\}$, one has
  $$
  \sup_{B_r(x_0)}u_p\geq\theta \; r,
  $$
  with the constant $\theta>0$ being independent of $p$.
\end{corollary}
\begin{proof}
  This is a direct consequence of Theorem \ref{t6.1} combined with Lemma \ref{l6.1}.
\end{proof}
The next result is from \cite{Y16} and reveals, as one may expect, that the positivity set is well localized inside a bounded set. This means that the optimization in $\R^n$ is actually in a large but bounded domain (Theorem 6.4 of \cite{Y16}).
\begin{lemma}\label{l6.2}
  If $u_p$ is a minimizer of \eqref{P}, then
  $$
  diam(\{u_p>0\})\leq diam(\Omega)+1+C(n)\gamma\frac{\|\varphi\|_{C^1}+1/\varepsilon}{\varepsilon}.
  $$
\end{lemma}

We are now ready to pass to the limit in \eqref{P} as $p\rightarrow\infty$.
\begin{theorem}\label{t6.2}
  If $u_p$ is a solution of \eqref{P} then, as $p\rightarrow\infty$ and up to a subsequence, $u_p\rightarrow u_\infty$ uniformly in $\R^n$ and weakly in every $W^{1,q}_0(\R^n)$, $q\in(1,\infty)$, where $u_\infty$ is a minimizer of \eqref{L}. In addition, $u_\infty$ grows linearly away from the free boundary, and is strongly non-degenerate, i.e.,
  $$
  u_\infty(x)\geq\theta \, \dist(x,\partial\{u_\infty>0\}),\,\,
  \forall x\in\{u_\infty>0\},
  $$
  and, for any $x_0\in\partial\{u_\infty>0\}$, one has
  $$
  \sup_{B_r(x_0)}u_\infty\geq\theta\,  r.
  $$
\end{theorem}
\begin{proof}
  Let $v\in\mathbb{K_\infty}$. Since $\mathbb{K}_\infty \subset \mathbb{K}_p$, for every $p>1$, and $u_p$ is a minimizer of \eqref{P}, then
  \begin{eqnarray}\label{6.5}
    \bigg(\int|\nabla u_p|^p\bigg)^{1/p}&\leq&\bigg(\int|\nabla v|^p\bigg)^{1/p}=\bigg(\int_{\Omega}|\nabla v|^p+\int_{\{v>0\}\setminus\Omega}|\nabla v|^p\bigg)^{1/p}\nonumber \\
    &\leq& \Lip(v)(|\Omega|+\gamma)^{1/p}\nonumber\\
    &\leq& C,
  \end{eqnarray}
  where $C>0$ is a constant independent of $p$. We now fix $q\in(1,\infty)$; dividing the integral, estimating each term and using \eqref{6.5}, we get
  \begin{eqnarray}\label{6.6}
    \bigg(\int|\nabla u_p|^q\bigg)^{1/q}  &=& \bigg(\int_\Omega |\nabla u_p|^q+\int_{\{u_p>0\}\setminus\Omega}|\nabla u_p|^q\bigg)^{1/q}\nonumber \\
     &\leq& \bigg(\bigg[\int_{\Omega}|\nabla u_p|^p\bigg]^{q/p}\big[|\Omega|^{p/(p-q)}+\gamma^{p/(p-q)}\big]\bigg)^{1/q}\nonumber \\
     &\leq& \Lip(v)(|\Omega|+\gamma)^{1/p}\bigg(|\Omega|^{p/(p-q)}+\gamma^{p/(p-q)}\bigg)^{1/q}\nonumber
     \\
     &\leq& C.
  \end{eqnarray}
 Hence, $u_p$ is uniformly bounded in $W^{1,q}(\R^n)$. Therefore, there is a weakly convergent subsequence and its limit $u_\infty$ satisfies (thanks to \eqref{6.6})
 \begin{equation*}
  \bigg(\int|\nabla u_\infty|^q\bigg)^{1/q}\leq\Lip(v)(|\Omega|+\gamma)^{1/q}\leq C.
 \end{equation*}
 Now taking $q\rightarrow\infty$ and performing a diagonal argument, one gets a subsequence (still denoted by $u_p$), which for every $q\in(1,\infty)$ converges weakly in $W^{1,q}(\R^n)$ to the function $u_\infty\in W^{1,\infty}(\R^n)$, and
 $$
\Lip (u_\infty) = \|\nabla u_\infty\|_{L^\infty(\R^n)}\leq\Lip(v).
 $$
 Clearly $u_\infty\geq\varphi$. As for estimating the Lebesgue measure of the positivity set, note that, for a fixed $\varepsilon>0$ and $p$ large enough, using the uniform convergence, one has
 $$
 \{u_\infty>\varepsilon\}\subset\{u_p>0\}.
 $$
 Therefore,
 $$
 |\{u_\infty>0\}\setminus\Omega|=
 \lim_{\varepsilon\rightarrow0}
 |\{u_\infty>\varepsilon\}\setminus\Omega| \leq \gamma.
 $$
 Consequently, $u_\infty$ is a minimizer for \eqref{L}. 
 
 It remains to check that the limit satisfies $\Delta_\infty u_\infty=0$ in $\{u_\infty>0\}\setminus\Omega$ and $\Delta_\infty u_\infty\leq0$ in $\Omega$, in the viscosity sense. Let $\phi\in C^2(\R^n)$ be such that $u_\infty-\phi$ has a local minimum at $x_0\in\R^n$ and $u_\infty(x_0)=\phi(x_0)$. Since $u_{p}\rightarrow u_\infty$ uniformly, then $u_{p}-\phi$ has a minimum at some point $x_p\in\R^n$ and $x_p\rightarrow x_0$. On the other hand, $\Delta_{p}u_{p}\leq0$ weakly, therefore (see \cite{MRU09} and \cite{RT12})
 \begin{equation}\label{6.7}
 (p-2)|\nabla\phi|^{p-4}\Delta_\infty\phi(x_p)+
 |\nabla\phi|^{p-2}\Delta\phi(x_p)\leq0.
 \end{equation}
 If $\nabla\phi(x_0)=0$ then $\Delta_\infty\phi(x_0)=0$. If $\nabla\phi(x_0)\neq0$, then $\nabla\phi(x_p)\neq0$ for large $p$, hence \eqref{6.7} implies
 $$
 \Delta_\infty\phi(x_p)\leq-
 \frac{1}{p-2}|\nabla\phi|^2\Delta\phi(x_p)
 \rightarrow0\,\textrm{ as }\,p\rightarrow\infty.
 $$
 That is, $\Delta_\infty\phi(x_0)\leq0$, which means that $\Delta_\infty u_\infty\leq0$ in the viscosity sense. When $x_0\in\{u_\infty>0\}\setminus\Omega$, then for large $p$ one has that $x_p$ is in the same positivity set, and moreover, we have $u_{p}>0$ in a neighborhood of $x_0$ and every $u_{p}$ is $p$-harmonic there, therefore (once again, see \cite{MRU09} and \cite{RT12}) one has the equality in \eqref{6.7}, which implies that $\Delta\phi(x_0)=0$. This means that $\Delta_\infty u_\infty=0$ in $\{u_\infty>0\}\setminus\Omega$ in the viscosity sense.

 The linear growth and strong non-degeneracy properties of $u_\infty$ follow from the fact that  $\displaystyle\lim p^{1/p}=1$. Passing to the limit as $p\rightarrow\infty$ in the corresponding inequalities in Lemma \ref{l6.1}, we get the desired results.
\end{proof}
The next theorem gives information about the free boundaries of the limiting problem.
\begin{theorem}\label{t6.3}
  If $u_p$ is a solution of \eqref{P} and $u_\infty$ is a solution of \eqref{L} then, as $p\rightarrow\infty$ and up to a subsequence,
  $$
  \partial\{u_p>0\}\longrightarrow\partial\{u_\infty>0\}
  $$
  and
  $$
  \partial \left( \{u_p>\varphi\} \cap \Omega \right) \longrightarrow\partial \left( \{u_\infty>\varphi\} \cap \Omega \right)  $$
  in Hausdorff distance.
\end{theorem}
\begin{proof}
  The proof is essentially the same as that of Theorem 5 in \cite{RT12}. We bring it here for the reader's convenience. Let $\Gamma_\varepsilon(E)$ be the $\varepsilon$-neighborhood of the set $E\subset\R^n$, i.e.,
  $$
  \Gamma_\varepsilon(E):=\{x\in\R^n\,:\,\dist(x,E)<\varepsilon\},
  \quad \varepsilon>0.
  $$
  We need to show that, for a given $\varepsilon>0$ and for $p$ large enough (depending on $\varepsilon$), one has
  $$
  \partial\{u_p>0\}\subset\Gamma_\varepsilon(\partial\{u_\infty>0\})
  $$
  and
  $$
  \partial\{u_\infty>0\}\subset\Gamma_\varepsilon(\partial\{u_p>0\}).
  $$
  It is enough to check the first inclusion, since the other one is proved similarly. Suppose the inclusion does not hold. It means that there is a point $z$ such that $z\in\partial\{u_p>0\}$ but $z\notin\Gamma_\varepsilon(\partial\{u_\infty>0\})$. The latter means that
  $$
  \dist(z,\partial\{u_\infty>0\})\geq\varepsilon.
  $$
  If $u_\infty(z)>0$ then, using Theorem \ref{t6.2}, one has
  $$
  u_\infty(z)\geq\theta\, \dist(z,\partial\{u_\infty>0\})
  \geq\theta \, \varepsilon.
  $$
  From the uniform convergence, we have $u_p(z)\geq\frac{2}{3}\theta\, \varepsilon$, for $p$ large enough, which contradicts the fact that $z\in\partial\{u_p>0\}$. Hence $u_\infty(z)=0$ and so $u_\infty\equiv0$ in $B_\varepsilon(z)$, which leads to a contradiction since, from Corollary \ref{c6.1}, one has 
  $$
  \sup_{B_{\varepsilon/2}(z)}u_p\geq\theta \, \frac{\varepsilon}{2}>0.
  $$
  
The convergence for the interior free boundaries is proved similarly.
\end{proof}

\bigskip

\noindent{\bf Acknowledgments.} This work was partially supported by FCT grant SFRH/BPD/92717/2013, and by the Centre for Mathematics of the University of Coimbra -- UID/MAT/00324/2013, funded by the Portuguese Government through FCT/MCTES and co-funded by the European Regional Development Fund through the Partnership Agree\-ment PT2020.


\begin{thebibliography}{99}

\bibitem{AAC86} N. Aguilera, H. Alt, L. Caffarelli, \textit{An optimization problem with volume constraint}, SIAM J. Control Optim. 24 (1986), 191--198.
\bibitem{ACS87} N. Aguilera, L. Caffarelli, J. Spruck, \textit{An optimization problem in heat conduction}, Ann. Scuola Norm. Sup. Pisa Cl. Sci (4) 14 (1987), 355--387.
\bibitem{AC81} H. Alt, L. Caffarelli, \textit{Existence and regularity for a minimum problem with free boundary}, J. Reine Angew. Math. 325 (1981), 105--144.
\bibitem{AT} M. Amaral, E.V. Teixeira, \textit{Free transmission problems}, Comm. Math. Phys. 337 (2015), 1465--1489. 
\bibitem{ALT} D. Ara\'ujo, R. Leit\~ao, E.V. Teixeira, \textit{Infinity Laplacian equation with strong absorptions}, J. Funct. Anal. 270 (2016), 2249--2267.
\bibitem{BMW06} J.F. Bonder, S. Mart\'{\i}nez, N. Wolanski, \textit{An optimization problem with volume constraint for a degenerate quasilinear operator}, J. Differential Equations 227 (2006), 80--101.
\bibitem{CS95} L. Caffarelli, S. Salsa, \textit{A geometric approach to free boundary problems}, Graduate Studies in Mathematics 68 (1995), AMS.
\bibitem{D83} E. DiBenedetoo, \textit{$C^{1,\alpha}$ local regularity of weak solutions of degenerate elliptic equations}, Nonlinear Anal. 7 (1983), 827--850.
\bibitem{KMV96} P. Koskela, J.J. Manfredi, E. Valliamor, \textit{Regularity theory and traces for $A$-harmonic functions}, Trans. Amer. Math. Soc. 348 (1996), 755--766.
\bibitem{LQT} R. Leit\~ao, O. de Queiroz, E.V. Teixeira, \textit{Regularity for degenerate two-phase free boundary problems}, Ann. Inst. H. Poincar\'e Anal. Non Lin\'eaire 32 (2015), 741--762. 
\bibitem{MRU09} J.J. Manfredi, J.D. Rossi, J.M. Urbano, \textit{$p(x)$-Harmonic functions with unbounded exponent in a subdomain}, Ann. Inst. H. Poincar\'e Anal. Non Lin\'eaire 26 (2009), 2581--2595.
\bibitem{OT06} K. Oliveira, E.V. Teixeira, \textit{An optimization problem with free boundary governed by a degenerate quasilinear operator}, Differential Integral Equations 19 (2006), 1061--1080.
\bibitem{PSU12} A. Petrosyan, H. Shahgholian, N. Uraltseva, \textit{Regularity of free boundaries in obstacle-type problems}, Graduate Studies in Mathematics 136 (2012), AMS.
\bibitem{RT12} J.D. Rossi, E.V. Teixeira, \textit{A limiting free boundary problem ruled by Aronsson's equation}, Trans. Amer. Math. Soc. 364 (2012), 703--719.
\bibitem{T05} E.V. Teixeira, \textit{The nonlinear optimization problem in heat conduction}, Calc. Var. Partial Differential Equations 24 (2005), 21--46.
\bibitem{T10} E.V. Teixeira, \textit{Optimal design problems in rough inhomogeneous media. Existence theory}, Amer. J. Math. 132 (2010), 1445--1492.
\bibitem{TT15} E.V. Teixeira, R. Teymurazyan, \textit{Optimal design problems with fractional diffusions}, J. Lond. Math. Soc. 92 (2015), 338--352.
\bibitem{T84} P. Tolksdorf, \textit{Regularity for a more general class of quasilinear elliptic equations}, J. Differential Equations 51 (1984), 126--150.
\bibitem{U68} N. Uraltseva, \textit{Degenerate quasilinear elliptic systems}, Zap. Naucn. Sem. Leningrad. Otdel. Mat. Inst. Steklov 7 (1968), 184--192 (in Russian).
\bibitem{Y16} H. Yu, \textit{An optimization problem in heat conduction with minimal temperature constraint, interior heating and exterior insulation}, Calc. Var. Partial Differential Equations 55 (2016), Art. 130, 15 pp.

 




\end{thebibliography}
\end{document}